\documentclass{amsart}
\usepackage[utf8]{inputenc}
\usepackage{amsmath}
\usepackage{amssymb}
\usepackage{amsthm}
\usepackage{enumitem}
\usepackage{todonotes}
\usepackage{hyperref}
\usepackage{url}
\usepackage{xcolor}
\usepackage{graphicx}
\usepackage{comment}

\setlist[enumerate]{label=\rm{(\arabic*)}, ref=(\arabic*)}

\newcommand{\N}{\mathbb{N}}
\newcommand{\Z}{\mathbb{Z}}

\newtheorem{thm}{Theorem}[section]
\newtheorem{lemma}[thm]{Lemma}

\newtheorem{cor}[thm]{Corollary}

\theoremstyle{definition}
\newtheorem{qu}[thm]{Question}
\newtheorem{defn}[thm]{Definition}

\newtheorem{notation}[thm]{Notation}

\newtheorem{thmx}{Theorem}[section]

\newtheorem{corx}[thmx]{Corollary}

\newtheorem*{MainThm}{Theorem \ref{thm:MainTechnical}}
\newtheorem*{corQuestion}{Corollary \ref{cor:FPGroupWithoutFree}}
\newtheorem*{corStronger}{Corollary \ref{cor:FPAmenable}}

\title{Intersection-saturated groups without free subgroups}
\author{Dominik Francoeur}

\begin{document}

\begin{abstract}
A group $G$ is said to be intersection-saturated if for every strictly positive integer $n$ and every map $c\colon \mathcal{P}(\{1,\dots, n\})\setminus \emptyset \rightarrow \{0,1\}$, one can find subgroups $H_1,\dots, H_n\leq G$ such that for every non-empty subset $I\subseteq \{1,\dots, n\}$, the intersection $\bigcap_{i\in I}H_i$ is finitely generated if and only if $c(I)=0$. We obtain a new criterion for a group to be intersection-saturated based on the existence of arbitrarily high direct powers of a subgroup admitting an automorphism with a non-finitely generated set of fixed points. We use this criterion to find new examples of intersection-saturated groups, including Thompson's groups and the Grigorchuk group. In particular, this proves the existence of finitely presented intersection-saturated groups without non-abelian free subgroups, thus answering a question of Delgado, Roy and Ventura.
\end{abstract}

\maketitle

\section{Introduction}

It is a classical result of Howson \cite{Howson54} that the intersection of finitely many finitely generated subgroups of a non-abelian free group is again finitely generated. This, however, does not hold for all groups (groups for which it does are said to possess the \emph{Howson property}).
Recently, Delgado, Roy and Ventura \cite{DelgadoRoyVentura23} introduced a new notion, \emph{intersection configurations}, that can be seen as a generalisation of these considerations.

\begin{defn}[cf. \cite{DelgadoRoyVentura23}]
Let $n\in \N_{\geq 1}$ be a strictly positive integer.
\begin{enumerate}
\item An $n$-\emph{configuration} is a map $c\colon \mathcal{P}(\{1,\dots, n\})\setminus \{\emptyset\}\rightarrow \{0,1\}$.
\item An $n$-configuration $c$ is said to be \emph{realisable} in a group $G$ if there exist subgroups $H_1,\dots, H_n\leq G$ such that for all non-empty subset $I\subseteq \{1,\dots, n\}$, the subgroup $\bigcap_{i\in I}H_i$ is finitely generated if and only if $c(I)=0$.
\end{enumerate}
\end{defn}
Thus, a group $G$ has the Howson property if and only if the $2$-configuration given by $c(\{1\})=c(\{2\})=0$ and $c(\{1,2\})=1$ is not realisable in $G$.

It is natural to ask which groups satisfy the property that all configurations are realisable. Such groups were called \emph{intersection-saturated} by Delgado, Roy and Ventura.

\begin{defn}[\cite{DelgadoRoyVentura23}, Definition 3.7]
A group $G$ is \emph{intersection-saturated} if for every $n\in \N_{\geq 1}$, every $n$-configuration is realisable.
\end{defn}

In \cite{DelgadoRoyVentura23}, Delgado, Roy and Ventura produced examples of intersection-saturated groups, all containing $F_2\times \Z^m$ for some $m\in \N$, where $F_2$ denotes the free groups on two generators. They asked (\cite{DelgadoRoyVentura23}, Question 7.2) if there exists a finitely presented intersection-saturated group not containing such a subgroup.

In the present note, we introduce a new technique to find intersection-saturated groups, which allows us to answer this question positively. More precisely, we obtain the following criterion.

\begin{thmx}\label{thm:MainTechnical}
%Let $K$ be a finitely generated group with an automorphism $f$ such that the subgroup of fixed points of $f$ is not finitely generated, and let $G$ be a group containing subgroups isomorphic to $K^n$ for all $n\in \N$, where $K^n$ denotes the direct product of $K$ with itself $n$ times. Then, $G$ is intersection-saturated.
Let $G$ be a group. Suppose that for every $n\in \N_{\geq 1}$, there exists a finitely generated group $K_n$ and an automorphism $f_n\colon K_n\rightarrow K_n$ such that
\begin{enumerate}[label=(\Roman*)]
\item the subgroup of fixed points of $f_n$ is not finitely generated,
\item $G$ contains a subgroup isomorphic to $K_n^n$, where $K_n^n$ denotes the direct product of $K_n$ with itself $n$ times.
\end{enumerate}
Then, $G$ is intersection-saturated.
\end{thmx}

We then apply this criterion to obtain several new examples of intersection-saturated groups, including Thompson's groups (Corollary \ref{cor:Thompson}) and the Grigorchuk group (Corollary \ref{cor:Grigorchuk}). These examples allow us to answer the question of Delgado, Roy and Ventura.

\begin{corx}\label{cor:FPGroupWithoutFree}
There exist finitely presented intersection-saturated groups without non-abelian free subgroups.
\end{corx}

In fact, we go one step further and prove that there are finitely presented intersection-saturated groups that are amenable.

\begin{corx}\label{cor:FPAmenable}
There exist finitely presented amenable intersection-saturated groups. 
\end{corx}

%Using a result of Rozhkov \cite{Rozhkov93}, we deduce that the group known as the \emph{Grigorchuk group} \cite{Grigorchuk80}, which is periodic and thus cannot contain free subgroups, is intersection-saturated.

%\begin{corx}\label{cor:WorksForGrigorchuk}
%The Grigorchuk group is intersection-saturated.
%\end{corx}
%
%Although the Grigorchuk group itself is not finitely presented (see \cite{delaHarpe00}, VIII.E Theorem 55), it can be embedded into a finitely presented group without non-abelian free subgroups by a result of Grigorchuk \cite{Grigorchuk98}, answering the question of Delgado, Roy and Ventura.

%\begin{corx}\label{cor:FPGroupWithoutFree}
%There exists a finitely presented intersection-saturated group without non-abelian free subgroups.
%\end{corx}

\subsection*{Acknowledgements}

The author would like to thank Jordi Delgado and Enric Ventura for bringing the question to his attention and for their valuable comments and suggestions on a previous version of this text. He is also grateful to Corentin Bodart for noticing a mistake in an earlier version. The author acknowledges the support of the Leverhulme Trust Research Project Grant RPG-2022-025.

\section{Proof of Theorem \ref{thm:MainTechnical}}

Let us first fix the notation that we will use for the rest of this note.

\begin{notation}
For $n\in \N_{\geq 1}$, we will write $[n]=\{1,\dots, n\}$.
\end{notation}

\begin{notation}
If $G$ is a group and $n\in \N_{\geq 1}$, we will denote by $G^{n}$ the direct product of $n$ copies of $G$.
For any $g\in G^n$ and $i\in [n]$, we will denote by $g_i$ the $i$\textsuperscript{th} component of $g$, so that $g=(g_1,g_2,\dots, g_n)$.
%For any $g\in G$ and $i\in \{1,2,\dots, n\}$, we will denote by $g^{(i)}\in G^n$ the element such that
%\[g_j^{(i)}=\begin{cases}
%g & \text{ if } i=j \\
%1 & \text{ otherwise}.
%\end{cases}\]
\end{notation}

\begin{notation}
Let $G$ be a group, let $n\in\N$ be any integer and let $H_1,\dots, H_n\leq G$ be a collection of subgroups. For $I\in \mathcal{P}([n])$, we will write
\[H_I = \bigcap_{i\in I}H_i.\]
\end{notation}

Let us now introduce the following construction, which will serve as a basis for the proof of the main result.

\begin{lemma}\label{lemma:SmallInteresctionsFGAllIntersectionsNFG}
Let $G$ be a finitely generated group and suppose that there exists an automorphism $f$ of $G$ such that the subgroup of fixed points of $f$ is not finitely generated.
Then, for any $n\in \N_{\geq 1}$, there exist finitely generated subgroups $H_1,\dots, H_n\leq G^n$ such that $H_I$ is finitely generated for any proper non-empty subset of $[n]$ but $H_{[n]}$ is not finitely generated.
\end{lemma}
\begin{proof}
If $n=1$, then it suffices to take any non-finitely generated subgroup of $G$, such as the set of fixed points of the automorphism $f$. Thus, let us now assume that $n>1$.
For each $1\leq i \leq n-1$, we define $H_i$ by
\[H_i=\left\{g\in G^n \mid g_i=g_{i+1}\right\},\]
and we define $H_n$ by
\[H_n = \left\{g\in G^n \mid g_n=f(g_1)\right\}.\]
Let $I\subset [n]$ be a proper subset.
We have
\[H_I = \left\{g\in G^n \mid g_i=g_{i+1} \forall i \in I\setminus \{n\} \text{ and }g_n=f(g_1) \text{ if } n\in I\right\}\]
If we write $k=|I|$, it is not hard to see that $H_I$ is isomorphic to $G^{n-k}$, since the coordinates corresponding to elements of $I$ are uniquely determined by the $n-k$ other coordinates, which have no restrictions placed upon them. As $G$ is finitely generated, it follows that $H_I$ must also be finitely generated.

Now, let us see that $H_{[n]}$ is not finitely generated.
We have
\begin{align*}
H_{[n]} &= \{g\in G^n \mid g_1=g_2=\dots = g_n = f(g_1)\}\\
&=\{(g,g,\dots, g)\in G^n \mid g=f(g)\}.
\end{align*}
Thus, we see that $H_{[n]}$ is isomorphic to the subgroup of fixed points of $f$ in $G$, which is not finitely generated by assumption.
\end{proof}

Using the previous lemma, we can realise all configurations taking value $1$ at most once in direct products, as the next lemma shows.

\begin{lemma}\label{lemma:1ConfigurationsRealisable}
Let $G$ be a finitely generated group with an automorphism $f$ such that the subgroup of fixed points of $f$ is not finitely generated. Then, for all $n\in \N_{\geq 1}$, every $n$-configuration $c\colon \mathcal{P}([n])\setminus \{\emptyset\}\rightarrow \{0,1\}$ such that $|c^{-1}(1)|\leq 1$ is realisable in $G^n$.
\end{lemma}
\begin{proof}
Let us fix some $n\in \N_{\geq 1}$ and a configuration $c$ taking the value $1$ at most once.
If $c$ takes only the value $0$, then one can simply set $H_1=H_2=\dots=H_n = 1$.
Let us now assume that there exists some non-empty set $I_1\subseteq [n]$ such that $c(I_1)=1$. By our assumptions on $c$, this set must be unique.
Let us write $I_1=\{j_1,\dots, j_{n_1}\}$, where $n_1=|I_1|$.

By Lemma \ref{lemma:SmallInteresctionsFGAllIntersectionsNFG}, there exist subgroups $H_{j_1},\dots, H_{j_{n_1}}\leq G^{n_1}\leq G^n$ such that $H_J$ is finitely generated for all proper subsets $J\subset I_1$ but $H_{I_1}$ is not finitely generated. Note that we have chosen here an arbitrary embedding of $G^{n_1}$ in $G^n$. For $i\notin I_1$, let us define $H_i = 1$. Then, for any $J\subseteq \{1,2,\dots, n\}$, we have one of the following three cases
\begin{enumerate}
\item $J=I_1$, in which case $H_J$ is not finitely generated,
\item $J\subsetneq I_1$, in which case $H_J$ is finitely generated,
\item $J\not\subset I_i$, in which case $H_J=1$ and is thus finitely generated.
\end{enumerate}
Thus, as desired, $H_J$ is finitely generated if and only if $c(J)=0$.
\end{proof}

To pass from configuration maps with at most one non-zero value to arbitrary configuration maps, let us recall the notion of the \emph{join} of two configurations, as defined by Delgado, Roy and Ventura \cite{DelgadoRoyVentura23}.

\begin{defn}[\cite{DelgadoRoyVentura23}, Definition 3.3]
Let $n\in \N_{\geq 1}$ and let $c_1,c_2$ be two $n$-configurations. Their \emph{join} is the $n$-configuration $c_1\wedge c_2 \colon \mathcal{P}(\{1,2,\dots, n\})\setminus \{\emptyset\}\rightarrow \{0,1\}$ defined by
\[c_1\wedge c_2 (I) = \begin{cases}
0 & \text{ if } c_1(I)=0 \text{ and } c_2(I) = 0 \\
1 & \text{ otherwise.}
\end{cases}\]
\end{defn}

Let us now see that the join of two configurations realisable in groups $G_1$ and $G_2$, respectively, is always realisable in the direct product $G_1\times G_2$.

\begin{lemma}\label{lemma:JoinRealisable}
Let $n\in \N_{\geq 1}$ and let $c_1,c_2$ be two $n$-configurations.
Let $G_1, G_2$ be two groups such that $c_1$ and $c_2$ are realisable configurations in $G_1$ and $G_2$, respectively.
Then, $c_1\wedge c_2$ is realisable in $G_1\times G_2$.
\end{lemma}
\begin{proof}
Let $H_1,\dots, H_n\leq G_1$ and $K_1,\dots, K_n \leq G_2$ be subgroups realising the configurations $c_1$ and $c_2$, respectively.
For all $i\in [n]$, let $L_i=H_i\times K_i \leq G_1\times G_2$. Then, we claim that $L_1, \dots, L_n$ are subgroups realising the configuration $c_1\wedge c_2$.
Indeed, let $I\subseteq [n]$ be a non-empty subset. Then, $L_I=H_I\times K_I$.
If both $H_I$ and $K_I$ are finitely generated, then $L_I$ is also finitely generated, but if one of $H_I$ or $K_I$ is not finitely generated, then $L_I$ cannot be either. This shows that $L_I$ exactly realises $c_1\wedge c_2$.
\end{proof}

We are now ready to prove Theorem \ref{thm:MainTechnical}, which we restate for the convenience of the reader.

\begin{MainThm}
Let $G$ be a group. Suppose that for every $n\in \N_{\geq 1}$, there exists a finitely generated group $K_n$ and an automorphism $f_n\colon K_n\rightarrow K_n$ such that
\begin{enumerate}[label=(\Roman*)]
\item the subgroup of fixed points of $f_n$ is not finitely generated,
\item $G$ contains a subgroup isomorphic to $K_n^n$.
\end{enumerate}
Then, $G$ is intersection-saturated.
\end{MainThm}
\begin{proof}
Let us fix $n\in \N_{\geq 1}$ and an $n$-configuration $c$. Let $k=|c^{-1}(1)|$, and let $I_1,\dots, I_k\in \mathcal{P}([n])\setminus \{\emptyset\}$ be an enumeration of the sets mapped to $1$ by $c$.
For every $1\leq i \leq k$, we define the $n$-configuration $c_i\colon \mathcal{P}([n])\setminus \{\emptyset\}\rightarrow \{0,1\}$ by $c_i(J)=1$ if and only if $J=I_i$. It is obvious that $c=c_1\wedge c_2\wedge \dots \wedge c_k$ (note that the join is associative, so that this expression is well-defined).

By Lemma \ref{lemma:1ConfigurationsRealisable}, for every $1\leq i \leq k$, the configuration $c_i$ is realisable in $K_{kn}^n$. Applying Lemma \ref{lemma:JoinRealisable} inductively, we then conclude that the configuration $c=c_1\wedge \dots \wedge c_k$ is realisable in $K_{kn}^{kn}$. Since $K_{kn}^{kn}$ embeds in $G$ by assumption, we conclude that the configuration $c$ is realisable in $G$.
\end{proof}

\section{New examples of intersection-saturated groups}

We will now apply Theorem \ref{thm:MainTechnical} to find new examples of intersection-saturated groups.
%
%We can apply Theorem \ref{thm:MainTechnical} to the Grigorchuk group, which yields Corollary \ref{cor:WorksForGrigorchuk}.
%
Our first application is Thompson's groups. We refer the reader to \cite{CannonFloydParry96} for an introduction to these groups.

\begin{cor}\label{cor:Thompson}
Thompson's groups $F$, $T$ and $V$ are intersection saturated.
\end{cor}
\begin{proof}
Since $F\leq T\leq V$, it suffices to prove the result for $F$.
By \cite{GubaSapir99} Corollary 22, the restricted wreath product $\Z\wr \Z = \bigoplus_{\Z}\Z \rtimes \Z$ is contained in $F$. Let $g\in \bigoplus_{\Z}\Z$ be any non-trivial element, and let $C(g)$ denote the centraliser of $g$ in $\Z\wr \Z$.
Since $\bigoplus_{\Z}\Z$ is abelian, it is clear that $\bigoplus_{\Z}\Z\leq C(g)$, and since the action of $\Z$ on $\bigoplus_{\Z}\Z$ has no fixed point except for the identity, we have in fact $\bigoplus_{\Z}\Z = C(g)$.

Let $f_g\colon \Z\wr \Z \rightarrow \Z \wr \Z$ denote conjugation by $g$.
Its set of fixed points is $C(g)=\bigoplus_{\Z}\Z$, which is not finitely generated.
It is known that for every $n\in \N$, $F$ contains a subgroup isomorphic to $F^n$ (this follows directly, for example, from \cite[Lemma 4.4]{CannonFloydParry96}).
Thus, $F$ contains also a subgroup isomorphic to $(\Z\wr \Z)^n$ for every $n\in \N$, and we conclude by Theorem \ref{thm:MainTechnical} that $F$ is intersection-saturated.
\end{proof}

Since Thompson's group $F$ is finitely presented and does not contain a non-abelian free subgroup \cite[Corollary 4.9]{CannonFloydParry96}, we immediately obtain the following corollary, which answers Question 7.2 of \cite{DelgadoRoyVentura23}.

\begin{corQuestion}
There exist finitely presented intersection-saturated groups without non-abelian free subgroups.
\end{corQuestion}

Since the amenability of Thompson's group $F$ is famously an open question, Corollary \ref{cor:FPGroupWithoutFree} still leaves open the question of the existence of finitely presented amenable intersection-saturated groups.
We will answer this question thanks to our second application of Theorem \ref{thm:MainTechnical}, which is about branch groups. We refer the reader to \cite{BartholdiGrigorchukSunic03} for the definition and an introduction to these groups.

\begin{thm}\label{thm:NiceBranchGroupsIntersectionSaturated}
Let $G$ be a finitely generated branch group.
If $G$ contains an element $\alpha\in G$ whose centraliser $C_G(\alpha)$ is not finitely generated, then $G$ is intersection-saturated.
\end{thm}
\begin{proof}
It follows from the definition of a branch group (see for example \cite[Definition 1.1]{BartholdiGrigorchukSunic03}) that for every $n\in \N$, there exists a subgroup $K_n\leq G$ such that
\begin{enumerate}
\item $K_n$ has $k_n$ distinct conjugates, for some $n\leq k_n <\infty$,\label{item:FinitelyManyConjugates}
\item for every $g\in G$, either $gK_ng^{-1}=K_n$ or $[K_n,gK_ng^{-1}]=K_n\cap gK_ng^{-1} = 1$,\label{item:ConjugatesInDirectProduct}
\item the subgroup $H_n=\langle \{gK_ng^{-1} \mid g\in G\}\rangle \cong K_n^{k_n}$ is normal and of finite index in $G$.\label{item:ProductRistIsNormalFiniteIndex}
\end{enumerate}
We note that since $G$ is finitely generated, condition \ref{item:ProductRistIsNormalFiniteIndex} implies that $K_n^{k_n}$, and thus $K_n$, must also be finitely generated. Therefore, to apply Theorem \ref{thm:MainTechnical}, it suffices to find for every $n\in \N$ an automorphism $f_n\colon K_n\rightarrow K_n$ whose set of fixed points is not finitely generated.

Let us now fix some $n\in \N$. By condition \ref{item:FinitelyManyConjugates}, there exists a bijection $\phi\colon [k_n]\rightarrow \{gK_ng^{-1} \mid g\in G\}$ between the set $[k_n]$ and the set of conjugates of $K_n$.
Since $G$ acts by conjugation on the set of conjugates of $K_n$, we can pull this action back through $\phi$ to obtain an action of $G$ on $[k_n]$. Let $l=\left|[k_n]/\langle \alpha \rangle\right|$ be the number of orbits in $[k_n]$ under the action of the subgroup of $G$ generated by $\alpha$, and for every $i\in [k_n]$, let $o(i) = |\langle \alpha \rangle \cdot i|$ be the size of the orbit of $i$ under the action of the subgroup generated by $\alpha$. Let $\{j_1,\dots, j_l\}\subseteq [k_n]$ be a set containing exactly one representative of each orbit under the action of $\langle\alpha\rangle$. We claim that $C_{H_n}(\alpha)$, the centraliser of $\alpha$ in $H_n$, is isomorphic to
\[L=\prod_{i=1}^{l}C_{\phi(j_i)}(\alpha^{o(j_i)}).\]

To see this, let us define a homomorphism $\psi\colon L \rightarrow H_n$ by setting
\[\psi(h_{j_i}) = \prod_{m=0}^{o(j_i)-1}\alpha^{m}h_{j_i}\alpha^{-m}\in \prod_{m=0}^{o(j_i)-1}\phi(\alpha^m\cdot j_i)\]
for $h_{j_i}\in C_{\phi(j_i)}(\alpha^{o(j_i)})\leq \phi(j_i)$ and then defining $\psi(h_{j_1}\dots h_{j_l}) = \psi(h_{j_1})\dots \psi({h_{j_l}})$. Using the fact that by condition \ref{item:ConjugatesInDirectProduct}, distinct conjugates commute and intersect trivially, it is easy to check that $\psi$ is an injective homomorphism. Furthermore, for $h_{j_i}\in C_{\phi(j_i)}(\alpha^{o(j_i)})$, we have
\[\alpha\psi(h_{j_i})\alpha^{-1} = \prod_{m=1}^{o(j_i)}\alpha^{m}h_{j_i}\alpha^{-m} = \prod_{m=0}^{o(j_i)-1}\alpha^{m}h_{j_i}\alpha^{-m} = \psi(h_{j_i}),\]
where we have used the fact that $\alpha^{o(j_i)}h_{j_i}\alpha^{-o(j_i)} = h_{j_i}$ since $h_{j_i}\in C_{\phi(j_i)}(\alpha^{o(j_i)})$. It follows that the image of $\psi$ is contained in $C_{H_n}(\alpha)$.

Now, let $h=h_1\cdots h_{k_n}\in C_{H_n}(\alpha)$ be any element of $H_n$ centralising $\alpha$. We have \[\alpha h \alpha^{-1} = (\alpha h_1 \alpha^{-1})\cdots (\alpha h_{k_n}\alpha^{-1}) = h_1 \cdots h_{k_n}. \]
Since, for any $j\in [k_n]$, we have $\alpha h_j \alpha^{-1} \in \phi(\alpha\cdot j)$, we must have $\alpha h_j \alpha^{-1} = h_{\alpha \cdot j}$ for all $j\in [k_n]$, which implies that $\alpha^{m}\cdot h_j \alpha^{-m} = h_{\alpha^{m}\cdot j}$.
In particular, for every $i\in [l]$, we have $\alpha^{o(j_i)} h_{j_i}\alpha^{-o(j_i)} = h_{\alpha^{o(j_i)\cdot j_i}} = h_{j_i}$, which means that $h_{j_i}\in C_{\phi(j_i)}(\alpha^{o(j_i)})$.
Since every $j\in [k_n]$ can be written as $\alpha^{m}\cdot h_{j_i}$ for some $i\in [l]$ and some $m\in \{0,\dots, o(j_i)-1\}$, we conclude from all this that $h=\psi(h_{j_1}\cdots h_{j_l})$.
This shows that $\psi(L) = C_{H_n}(\alpha)$ and thus finishes showing that $C_{H_n}(\alpha)\cong L$.

By assumption, $C_G(\alpha)$ is not finitely generated. Since $H_n$ is of finite index in $G$, $C_{H_n}(\alpha) = C_G(\alpha)\cap H_n$ must be of finite index in $C_G(\alpha)$ and thus cannot be finitely generated.
This implies that there exists some $i_0\in [l]$ such that $C_{\phi(j_{i_0})}(\alpha^{o(j_{i_0})})$ is not finitely generated.
Indeed, we have just seen that $C_{H_n}(\alpha)\cong \prod_{i=1}^{l}C_{\phi(j_i)}(\alpha^{o(j_i)})$, so if all $C_{\phi(j_i)}(\alpha^{o(j_i)})$ were finitely generated, $C_{H_n}(\alpha)$ would be as well.
Notice that by construction, $\alpha^{o(j_{i_0})}$ normalises $\phi(j_{i_0})$, so that conjugation by $\alpha^{o(j_{i_0})}$ is an automorphism of $\phi(j_{i_0})$ whose set of fixed points, $C_{\phi(j_{i_0})}(\alpha^{o(j_{i_0})})$, is not finitely generated.
Since $\phi(j_{i_0})$ is by definition a conjugate of $K_n$, and thus isomorphic to it, we have just proven the existence of an automorphism $f_n\colon K_n \rightarrow K_n$ whose set of fixed points is not finitely generated.
This finishes proving that the assumptions of Theorem \ref{thm:MainTechnical} are satisfied by $G$ and thus that $G$ is intersection-saturated.
\end{proof}

As a corollary, we get that the Grigorchuk group is intersection-saturated (we refer the reader to \cite{delaHarpe00} for an introduction to this group).

\begin{cor}\label{cor:Grigorchuk}
The Grigorchuk group is intersection-saturated.
\end{cor}
\begin{proof}
%Let $\Gamma$ be the Grigorchuk group, which is a finitely generated infinite periodic group (see \cite{delaHarpe00} for an introduction to this group). There exists a finite index normal subgroup $K\leq \Gamma$ such that $K\times K$ embeds into $K$ as a subgroup of finite index (see for instance \cite{delaHarpe00}, VIII.C.30). Therefore, by induction, $K^n$ embeds into $\Gamma$ for all $n\in \N_{\geq 1}$. Furthermore, as $K$ is of finite index in $\Gamma$, it is finitely generated. Thus, to apply Theorem \ref{thm:MainTechnical}, it suffices to find an automorphism of $K$ whose subgroup of fixed points is not finitely generated.
%
%Since $K$ is a normal subgroup of $\Gamma$, any element $\gamma\in \Gamma$ defines an automorphism $f_{\gamma}\colon K\rightarrow K$ by conjugation. An element $g\in K$ is fixed by such an automorphism $\gamma$ if and only if $g\in C(\gamma)\cap K$, where $C(\gamma)$ denotes the centraliser of $\gamma$ in $\Gamma$.
%By a theorem of Rozhkov (\cite{Rozhkov93}, Theorem 1), the group $\Gamma$ admits elements whose centralisers are not finitely generated. If $\gamma_0\in \Gamma$ is such an element, then $C(\gamma_0)\cap K$ cannot be finitely generated either, as it is of finite index in $C(\gamma_0)$. We can thus apply Theorem \ref{thm:MainTechnical} to $\Gamma$.
%
The Grigorchuk group is a finitely generated branch group (see for example \cite[Proposition 1.25]{BartholdiGrigorchukSunic03}), and by a theorem of Rozhkov \cite[Theorem 1]{Rozhkov93}, it admits elements whose centralisers are not finitely generated. Thus, by Theorem \ref{thm:NiceBranchGroupsIntersectionSaturated}, it is intersection-saturated.
\end{proof}

Although we consider this to be out of the scope of the current article, we believe that it should be fairly straightforward to adapt Rozhkov's arguments in \cite{Rozhkov93} to show that all finitely generated branch \emph{spinal groups} possess elements whose centralisers are not finitely generated, and thus are intersection-saturated (see \cite{BartholdiGrigorchukSunic03} for the definition of spinal groups). We do not currently know, however, this must be the case for all finitely generated branch groups.

\begin{qu}
Do all finitely generated branch groups possess an element whose centraliser is not finitely generated?
\end{qu}
If the answer to the above question is negative, could there exist finitely generated branch groups that are not intersection-saturated?
\begin{qu}
Are all finitely generated branch groups intersection-saturated?
\end{qu}

Let us finally conclude this note by using Corollary \ref{cor:Grigorchuk} to show that there exist finitely presented amenable intersection-saturated groups.

\begin{corStronger}
There exist finitely presented amenable intersection-saturated groups. 
\end{corStronger}
\begin{proof}
By a result of Grigorchuk (\cite{Grigorchuk98}, Theorem 1), there exists a finitely presented amenable group $\overline{\Gamma}$ containing the Grigorchuk group $\Gamma$ as a subgroup. Since $\Gamma$ is intersection-saturated by Corollary \ref{cor:Grigorchuk}, $\overline{\Gamma}$ is intersection-saturated.
\end{proof}

\bibliographystyle{plain}
\bibliography{biblio}

\end{document}